\documentclass[11pt,reqno]{amsart}
\newtheorem{theorem}{Theorem}[section]
\newtheorem{definition}{Definition}[section]
\newtheorem{lemma}{Lemma}[section]

\theoremstyle{definition}
\newtheorem{remark}{Remark}[section]
\newtheorem{example}{Example}[section]

\newcommand{\be}{\begin{equation}}
\newcommand{\ee}{\end{equation}}
\newcommand{\bea}{\begin{eqnarray}}
\newcommand{\eea}{\end{eqnarray}}
\newcommand{\beb}{\begin{eqnarray*}}
\newcommand{\eeb}{\end{eqnarray*}}
\newcommand{\norm}[1]{\left\lVert#1\right\rVert}
\usepackage{amssymb,amsfonts,amsthm,setspace,indentfirst,mathrsfs}
\usepackage{minibox}
\usepackage{enumitem}
\usepackage{tikz-cd}
\usepackage{centernot}
\usepackage{stmaryrd}
\usepackage{mathrsfs}
\usepackage{hyperref}
\numberwithin{equation}{section}

\begin{document}
\title[Deferred statistical convergence  in PNS]{Certain aspects of deferred statistical convergence of sequences in probabilistic normed spaces}

\author[N. Hossain, R. Mondal]{$^1$Nesar Hossain, $^2$Rahul Mondal}

\address{$^1$Department of Basic Science and Humanities, Dumkal Institute of Engineering and Technology, West Bengal-742406, India.}
\address{$^2$Department of Mathematics, Vivekananda Satavarshiki Mahavidyalaya, Manikpara, Jhargram -721513, West Bengal, India.}

\email{nesarhossain24@gmail.com$^1$}
\email{imondalrahul@gmail.com$^2$}


\begin{abstract}
In this research article, we have primarily focused on the circumstantial investigation  of  deferred statistical convergence of sequences and investigated some fundamental results compatible with the structure of a probabilistic normed space. Additionally, the idea of deferred statistical Cauchy sequences has been discussed with reference to the structure of a probabilistic normed space.
\end{abstract}

\maketitle

\section{Introduction}
\noindent The concept of statistical convergence of real sequences was first introduced by Fast \cite{Fast} and Steinhaus \cite{Steinhaus} independently in 1951. It is a generalization of the traditional notion of convergence for sequences, incorporating the concept of natural density. Introduced to address sequences that may not converge in the classical sense, it allows for a more nuanced understanding of the behavior of sequences over the long term. Statistical convergence is particularly useful in various fields of mathematics and applied sciences, such as functional analysis and number theory, where classical convergence may be too restrictive. It allows for the study of sequences that exhibit convergence behavior in a "statistical" sense, accommodating irregularities and providing a robust framework for analysis. Later, many works \cite{Baliarsingh, Connor, Dagadur, Debnath, Debnath2022, Et, Ercan, Fridy, Melliani, Nuray, Srivastava, Ulusu} in the literature built upon the ideas of Fast \cite{Fast} and Steinhaus \cite{Steinhaus}. \\ 
\indent In $1942$, Menger \cite{Menger} first proposed the concept of statistical metric space, now called probabilistic metric space. A probabilistic metric space is a generalization of a metric space where the concept of distance between points is extended to account for uncertainty or variability. Instead of assigning a single numerical value to the distance between two points, a probabilistic metric space assigns a probability distribution. This approach is useful in various fields such as probability theory, statistics, and fuzzy set theory, where distances between points may not be precisely defined. Later, Schweizer and Sklar \cite{Schweizer} expanded upon the idea introduced by Menger \cite{Menger}. The ideas of statistical metric space and normed linear space were merged by \v{S}erstnev \cite{Serstnev} into the concept of probabilistic normed space. In $1993$ Alsina et al. \cite{Alsina} gave a new definition of probabilistic normed space which is basically a special case of the definition of \v{S}erstnev.\\
\indent In $2016$,  K\"{u}\c{c}\"{u}kaslan et al.\cite{Kucukaslan} introduced the idea of deferred statistical convergence of sequences using the notion of deferred density, which is basically an extension of statistical convergence, lacunary statistical convergence and $\lambda$-statistical convergence. Deferred statistical convergence is a refinement of the concept of statistical convergence that allows for a more flexible understanding of sequence behavior by considering a ``deferred" or adjusted convergence criterion. This concept allows for a more nuanced analysis of sequences by accommodating a margin of error or fluctuation, thus offering greater flexibility in understanding convergence behavior. Deferred statistical convergence is particularly useful in settings where data or sequences may exhibit inherent variability, providing a way to study their long-term behavior while accounting for such variations.\\
\indent Here, we have shown that a usual convergent sequence must also be a strong deferred convergent sequence. Additionally, it has been clarified that the converse may not be true in general. The idea of deferred statistically Cauchy sequence has been discussed with an example. It has been shown that a strong deferred convergent sequence must be a deferred statistically Cauchy sequence. We have also discussed some important properties based on the works of Debnath \cite{Debnath2022}, Kucukaslan \cite{Kucukaslan} and Melliani \cite{Melliani}.  
\section{Preliminaries}
Throughout the paper $\mathbb{N}$ and $\mathbb{R}$ denote the set of natural numbers and the set of reals respectively. $|\mathscr{A}|$ denotes the cardinality of the set $\mathscr{A}$. $\mathscr{A}^c$ denotes the complement of $\mathscr{A}$ in $\mathbb{N}$. First we recall some basic definitions and notations.

\begin{definition}
Let $\mathscr{K}\subset \mathbb{N}$. Then the natural density of $\mathscr{K}$, denoted by $\delta(\mathscr{K})$,  is defined  as $$\delta(\mathscr{K})=\lim_{n\to\infty}\frac{1}{n}\vert\{k\leq n: k\in \mathscr{K}\}\vert,$$ provided the limit exists, where the vertical bars denote the cardinality of the enclosed set.
\end{definition}

In $1932$ Agnew \cite{Agnew} defined the deferred Ces\`{a}ro mean as a generalization of Ces\`{a}ro mean of real (complex) valued sequence $\{w_k\}$ by $$\mathscr{D}_\alpha^\vartheta(w_k)=\frac{1}{\vartheta(n)-\alpha(n)}\sum_{k=\alpha(n)+1}^{\vartheta(n)} w_k, \ k=1,2,3,\ldots$$ where $\alpha(n)$ and $\vartheta(n)$ are the sequences of non-negative integers satisfying $\alpha(n)<\vartheta(n)$ and $\lim_{n\to\infty}\vartheta(n)=\infty$.

\begin{definition}\cite{Kucukaslan}
    A sequence $\{w_k\}$ is said to be strong $\mathscr{D}_\alpha^\vartheta$-convergent to $\xi$ if $$\lim_{n\to\infty}\frac{1}{\vartheta(n)-\alpha(n)}\sum_{k=\alpha(n)+1}^{\vartheta(n)}\vert w_k-\xi\vert=0.$$ In this case, we write $w_k\xrightarrow{\mathscr{D}_\alpha^\vartheta}\xi$.
\end{definition}

\begin{definition}\cite{Yilmazturk}
    Let $\mathscr{K}\subset\mathbb{N}$ and $\mathscr{K}_{\alpha,\vartheta}(n)$ denotes the set $\{\alpha(n)+1\leq k\leq \vartheta(n): k\in\mathscr{K}\}$ where $\alpha(n)$ and $\vartheta(n)$ are the sequences of non-negative integers satisfying $\alpha(n)<\vartheta(n)$ and $\lim_{n\to\infty}\vartheta(n)=\infty$. The deferred density of $\mathscr{K}$ denoted by $\delta_\alpha^\vartheta(\mathscr{K})$ and defined by $\delta_\alpha^\vartheta(\mathscr{K})=\lim_{n\to\infty}\frac{1}{\vartheta(n)-\alpha(n)}\vert\mathscr{K}_{\alpha,\vartheta}(n)\vert$.
\end{definition}

\begin{definition}\cite{Kucukaslan}
    A sequence $\{w_k\}$ is said to be deferred statistically convergent to $\xi$ if for every $\varepsilon>0$, $$\lim_{n\to\infty}\frac{1}{\vartheta(n)-\alpha(n)}\vert\{\alpha(n)+1\leq k\leq \vartheta(n): \vert w_k-\xi\vert\geq \varepsilon \}\vert=0$$.
\end{definition}

\begin{definition}\cite{Schweizer}
 A binary operation $\odot : \mathscr{J}\times \mathscr{J}\rightarrow \mathscr{J}$, where $\mathscr{J}=[0,1]$ is named to be a continuous $t$-norm if for each $\nu_1,\nu_2,\nu_3,\nu_4\in \mathscr{J}$, the below conditions hold:
\begin{enumerate}
    \item $\odot$ is associative and commutative;
    \item  $\odot$ is continuous;
    \item $\nu_1\odot 1=\nu_1$ for all $\nu_1\in \mathscr{J}$;
    \item $\nu_1\odot \nu_2\leq \nu_3\odot \nu_4$ whenever $\nu_1\leq \nu_3$ and $\nu_2\leq \nu_4$. 
\end{enumerate}
\end{definition}

\begin{example}\cite{Klement}
  The following are the examples of $t$-norms:
  \begin{enumerate}
      \item $\nu_1\odot \nu_2= min\{\nu_1,\nu_2\}$;
      \item $\nu_1\odot \nu_2=\nu_1.\nu_2$;
      \item $\nu_1\odot \nu_2= max\{\nu_1+\nu_2-1,0\}$. This $t$-norm is known as Lukasiewicz $t$-norm.
  \end{enumerate}
\end{example}

\begin{definition}\cite{Frank}
  A function $\mathscr{G}: \mathbb{R}\rightarrow \mathbb{R}_0^+$  is said to be a distribution function if it is non decreasing and left continuous with $\inf_{t\in\mathbb{R}} \mathscr{G}(t)=0$ and $\sup_{t\in\mathbb{R}} \mathscr{G}(t)=1$. We denote $\Delta$ as the set of all distribution functions.
\end{definition}

\begin{definition}\cite{Frank}
      Let $\mathscr{X}$ be a real vector space, $\varphi$ be a mapping from $\mathscr{X}$ into $\Delta$ where  $\tau\in \mathscr{X}, \varepsilon\in \mathbb{R}$,   the value $\varphi(\tau)(\varepsilon)$ of the distribution function $\varphi(\tau)$ at $\varepsilon$ is denoted by $\varphi(\tau;\varepsilon)$ and $\varphi$ be a $t$-norm satisfying the following conditions:
    \begin{enumerate}
        \item $\varphi(\tau;0)=0$;
        \item $\varphi(\tau;\varepsilon)=1$, $\forall \ \varepsilon>0$ iff $\tau=\theta$, $\theta$ being the zero element of $\mathscr{X}$;
        \item $\varphi(\kappa \tau;\varepsilon)=\varphi(\tau;\frac{\varepsilon}{\vert\kappa\vert})$, $\forall \ \kappa\in\mathbb{R}\setminus \{0\}$ and $\forall \ \varepsilon>0$;
        \item $\varphi(\tau+\zeta;\varepsilon+\lambda)\geq \varphi(\tau;\varepsilon)\odot\varphi(\zeta;\lambda)$, $\forall \ \tau,\zeta\in \mathscr{X}$ and $\forall\ \varepsilon,\lambda\in \mathbb{R}_0^+$.
    \end{enumerate}
    Then the  triplet $(\mathscr{X},\varphi,\odot)$ is called a probabilistic normed space (shortly PNS).
\end{definition}

\begin{definition}\cite{Karakus}
Let $\{w_k\}$ be a sequence in a PNS $(\mathscr{X},\varphi,\odot)$. Then $\{w_k\}$ is named to be convergent to $\xi\in \mathscr{X}$ with respect to the probabilistic norm  $\varphi$ if for every $\varepsilon>0$ and $\sigma\in(0,1)$, there is a positive integer $n_0$ such that $\varphi(w_k-\xi;\varepsilon)>1-\sigma$ for all $k\geq n_0$. In this case we write $w_k\xrightarrow{\varphi}\xi\ \text{or}\ \varphi\text{-}\lim w_k=\xi$.   
\end{definition}

\begin{remark}\label{rem2.1}\cite{Aghajani}
    For a real normed space $(\mathscr{X},\norm{\cdot})$, we define the probabilistic norm $\varphi_0$ for $\tau\in \mathscr{X}, \varepsilon
    >0$ as $\varphi_0(\tau;\varepsilon)=\frac{\varepsilon}{\varepsilon+\norm{\tau}}$. Then   $w_k\xrightarrow{\norm{\cdot}}\xi$ if and only if $w_k\xrightarrow{\varphi_0}\xi$.
\end{remark}

\begin{definition}\cite{Karakus}
    Let $\{w_k\}$ be a sequence in a PNS $(\mathscr{X},\varphi,\odot)$. Then $\{w_k\}$ is named to be statistically convergent to $\xi\in \mathscr{X}$ with respect to the probabilistic norm  $\varphi$ if for every $\varepsilon>0$ and $\sigma\in(0,1)$, $\delta(\{k\in\mathbb{N}: \varphi(w_k-\xi;\varepsilon)\leq 1-\sigma \})=0$ or equivalently $\lim_{n\to\infty}\frac{1}{n}\vert\{k\leq n:  \varphi(w_k-\xi;\varepsilon)\leq 1-\sigma\}\vert=0$. In this scenario, we write $\mathcal{S}^\varphi-\lim w_k=\xi$ or $w_k\xrightarrow{\mathcal{S}^\varphi}\xi$.
\end{definition}

Throughout the work we will use some ideas that can be found in \cite{r1, r2, r3, r4, r5}.


\section{Main Results}
Throughout this section $\mathscr{X}$ will stand for probabilistic normed space. And, $\alpha(n)$ and $\vartheta(n)$ are the sequences of non-negative integers satisfying $\alpha(n)<\vartheta(n)$ and $\lim_{n\to\infty}\vartheta(n)=\infty$. $\delta_\alpha^\vartheta(\mathscr{K})$ stands for the deferred density of the set $\mathscr{K}$. $[x]$ denotes the greatest integer function. First we define the following:
\begin{definition}
    Let $\{w_k\}$ be a sequence in a PNS $\mathscr{X}$. Then $\{w_k\}$ is named to be strong $\mathscr{D}_\alpha^\vartheta$-convergent to $\xi\in\mathscr{X}$ with respect to probabilistic norm $\varphi$ if for every $\varepsilon>0$ and $\sigma\in(0,1)$ there exists $n_0\in\mathbb{N}$ such that $\frac{1}{\vartheta(n)-\alpha(n)}\sum_{k=\alpha(n)+1}^{\vartheta(n)}\varphi(w_k-\xi;\varepsilon)>1-\sigma$ for all $k>n_0$. In this scenario, we write $\mathscr{D}_\alpha^\vartheta(\varphi)-\lim w_k=\xi$ or  $w_k\xrightarrow{\mathscr{D}_\alpha^\vartheta(\varphi)}\xi$.
\end{definition}

\begin{definition}\label{defi3.2}
    Let $\{w_k\}$ be a sequence in a PNS $\mathscr{X}$. Then $\{w_k\}$ is named to be deferred statistically convergent to $\xi\in\mathscr{X}$ with respect to probabilistic norm $\varphi$ if for every $\varepsilon>0$ and $\sigma\in(0,1)$, $\delta_\alpha^\vartheta(\{\alpha(n)+1\leq k\leq \vartheta(n): \varphi(w_k-\xi;\varepsilon)\leq 1-\sigma \})=0$ or equivalently $\lim_{n\to\infty}\frac{1}{\vartheta(n)-\alpha(n)}\vert\{\alpha(n)+1\leq k\leq \vartheta(n): \varphi(w_k-\xi;\varepsilon)\leq 1-\sigma \}|=0$. In this scenario, we write $\mathscr{D}_\alpha^\vartheta(\mathcal{S}^\varphi)-\lim w_k=\xi$  or $w_k\xrightarrow{\mathscr{D}_\alpha^\vartheta(\mathcal{S}^\varphi)}\xi$.
\end{definition}

\begin{remark}
    From Definition \ref{defi3.2}, we have the following:
    \begin{enumerate}
        \item If $\alpha(n)=0$ and $\vartheta(n)=n$ then the notion of deferred statistical convergence coincides with the notion of statistical convergence \cite{Karakus} with respect to probabilistic norm $\varphi$.
        \item If $t(n)=\lambda_n$ and $s(n)=0$ then  the notion of deferred statistical convergence coincides with the notion of $\lambda$-statistical convergence \cite{Alotaibi} with respect to probabilistic norm $\varphi$.
        \item If $\vartheta(n)=k_r$ and $\alpha(n)=k_{r-1}$ then the notion of deferred statistical convergence coincides with the notion of lacunary statistical convergence \cite{Rahmat} with respect to probabilistic norm $\varphi$.
    \end{enumerate}
\end{remark}

\begin{example}
    Let $\mathscr{X}$ be a real normed space equipped with the usual norm. Define the continuous $t$-norm $\nu_1\odot \nu_2=\nu_1\nu_2$ for all $\nu_1,\nu_2\in[0,1]$. We take $\varphi_0(\tau;\varepsilon)=\frac{\varepsilon}{\varepsilon+\vert\tau\vert}$ for $\tau\in\mathscr{X}, \varepsilon>0$. Then $\varphi_0$ is a probabilistic norm on $\mathscr{X}$. Define the sequence $\{w_k\}$ as $w_k=\begin{cases}
        k^2, \ &\text{if}\ \left[\left \vert \sqrt{\vartheta(n)}\right\vert\right]-k_0<k\leq \left[\left\vert \sqrt{\vartheta(n)}\right\vert\right]\\
        0, \ &\text{otherwise}
    \end{cases}$. 
    Where $0<\alpha(n)\leq \left[\left\vert \sqrt{\vartheta(n)}\right\vert\right]-k_0$ where $k_0\in\mathbb{N}$ is a fixed number and $\vartheta(n)$ is a monotonic increasing sequence.
    Let $\xi=0$. Then for any $\varepsilon>0$ and $\sigma\in(0,1)$ we have \begin{align*}
        \mathscr{A}(\sigma,\varepsilon)&=\{\alpha(n)+1\leq k\leq \vartheta(n): \varphi_0(w_k-\xi;\varepsilon)\leq 1-\sigma\}\\
        &=\{\alpha(n)+1\leq k\leq \vartheta(n): \frac{\varepsilon}{\varepsilon+\vert w_k\vert}\leq 1-\sigma\}\\
        &=\{\alpha(n)+1\leq k\leq \vartheta(n): \vert w_k\vert \geq \frac{\varepsilon\sigma}{1-\sigma}>0\}\\
        &\subseteq\{\alpha(n)+1\leq k\leq \vartheta(n): w_k=k^2 \}.
    \end{align*}
    Therefore, $\delta_\alpha^\vartheta(\mathscr{A}(\sigma,\varepsilon))=\lim_{n\to\infty}\frac{\vert\mathscr{A}(\sigma,\varepsilon)\vert}{\vartheta(n)-\alpha(n)}\leq \lim_{n\to\infty}\frac{k_0}{\vartheta(n)-\alpha(n)}=0$. This gives that $w_k\xrightarrow{\mathscr{D}_\alpha^\vartheta(\mathcal{S}^\varphi)}\xi$.
\end{example}

\begin{lemma}
     Let $\{w_k\}$ be a sequence in a PNS $\mathscr{X}$. Then for every $\varepsilon>0$ and $\sigma\in(0,1)$, the following statements are gratified:
     \begin{enumerate}
         \item $\mathscr{D}_\alpha^\vartheta(\mathcal{S}^\varphi)-\lim w_k=\xi$;
         \item $\delta_\alpha^\vartheta(\{\alpha(n)+1\leq k\leq \vartheta(n): \varphi(w_k-\xi;\varepsilon)\leq 1-\sigma \})=0$;
         \item $\delta_\alpha^\vartheta(\{\alpha(n)+1\leq k\leq \vartheta(n): \varphi(w_k-\xi;\varepsilon)> 1-\sigma \})=1$;
         \item $\mathscr{D}_\alpha^\vartheta(\mathcal{S}^\varphi)-\lim \varphi(w_k-\xi;\varepsilon)=1$.
     \end{enumerate}
\end{lemma}

\begin{theorem}
     Let $\{w_k\}$ be a sequence in a PNS $\mathscr{X}$. If $\{w_k\}$ is deferred statistically convergent then $\mathscr{D}_\alpha^\vartheta(\mathcal{S}^\varphi)$-limit of $\{w_k\}$ is unique.
\end{theorem}

\begin{proof}
    If possible, let $\mathscr{D}_\alpha^\vartheta(\mathcal{S}^\varphi)-\lim w_k=\xi$ and $\mathscr{D}_\alpha^\vartheta(\mathcal{S}^\varphi)-\lim w_k=\beta$ where $\xi\neq \beta$. Now, for a given $\sigma \in (0,1)$ choose $\lambda\in(0,1)$ such that $(1-\lambda)\odot(1-\lambda)>1-\sigma$. Also, for any $\varepsilon>0$ we consider the sets $$\mathscr{A}(\lambda,\varepsilon)=\{\alpha(n)+1\leq k\leq \vartheta(n): \varphi(w_k-\xi;\frac{\varepsilon}{2})\leq 1-\lambda \}$$ and $$\mathscr{B}(\lambda,\varepsilon)=\{\alpha(n)+1\leq k\leq \vartheta(n): \varphi(w_k-\beta;\frac{\varepsilon}{2})\leq 1-\lambda \}.$$ By our assumption, $\delta_\alpha^\vartheta(\mathscr{A}(\lambda,\varepsilon))=0$ and $\delta_\alpha^\vartheta(\mathscr{B}(\lambda,\varepsilon))=0$. Hence $\delta_\alpha^\vartheta(\mathbb{N}\setminus \left[\mathscr{A}(\lambda,\varepsilon)\cap  \mathscr{B}(\lambda,\varepsilon) \right])=1$. Let $m\in \mathbb{N}\setminus \left[\mathscr{A}(\lambda,\varepsilon)\cap  \mathscr{B}(\lambda,\varepsilon) \right]$. So, we have 
    \begin{align*}
        &\varphi(\xi-\beta;\varepsilon)\\
        \geq &\varphi(w_m-\xi;\frac{\varepsilon}{2})\odot\varphi(w_m-\beta;\frac{\varepsilon}{2})\\
        >&(1-\lambda)\odot(1-\lambda)>1-\sigma.
    \end{align*}
    Since $\sigma\in(0,1)$ is arbitrary, $\varphi(\xi-\beta;\varepsilon)=1$ which yields $\xi=\beta$. This completes the proof.
\end{proof}

Now we proceed with the algebraic characterization of $\mathscr{D}_\alpha^\vartheta(\mathcal{S}^\varphi)$-convergence of $\{w_k\}$.
\begin{theorem}
     Let $\{w_k\}$ and $\{l_k\}$ be two sequences in a PNS $\mathscr{X}$. Then we have 
     \begin{enumerate}
         \item If $\mathscr{D}_\alpha^\vartheta(\mathcal{S}^\varphi)-\lim w_k=\xi$ and $\mathscr{D}_\alpha^\vartheta(\mathcal{S}^\varphi)-\lim l_k=\beta$ then $\mathscr{D}_\alpha^\vartheta(\mathcal{S}^\varphi)-\lim w_k+l_k=\xi+\beta$;
         \item If $\mathscr{D}_\alpha^\vartheta(\mathcal{S}^\varphi)-\lim w_k=\xi$ then $\mathscr{D}_\alpha^\vartheta(\mathcal{S}^\varphi)-\lim \kappa w_k=\kappa\xi$ where $\kappa\in\mathbb{R}$.
     \end{enumerate}
\end{theorem}

\begin{proof}
    \begin{enumerate}
        \item Let $\mathscr{D}_\alpha^\vartheta(\mathcal{S}^\varphi)-\lim w_k=\xi$ and $\mathscr{D}_\alpha^\vartheta(\mathcal{S}^\varphi)-\lim l_k=\beta$.  For a given $\sigma \in (0,1)$ choose $\lambda\in(0,1)$ such that $(1-\lambda)\odot(1-\lambda)>1-\sigma$. Then, for any $\varepsilon>0$, $\delta_\alpha^\vartheta(\mathscr{A}(\lambda,\varepsilon))=0$ and $\delta_\alpha^\vartheta(\mathscr{B}(\lambda,\varepsilon))=0$ where $$\mathscr{A}(\lambda,\varepsilon)=\{\alpha(n)+1\leq k\leq \vartheta(n): \varphi(w_k-\xi;\frac{\varepsilon}{2})\leq 1-\lambda \}$$ and $$\mathscr{B}(\lambda,\varepsilon)=\{\alpha(n)+1\leq k\leq \vartheta(n): \varphi(l_k-\beta;\frac{\varepsilon}{2})\leq 1-\lambda \}.$$ Let $$\mathscr{C}(\sigma,\varepsilon)=\{\alpha(n)+1\leq k\leq \vartheta(n): \varphi(w_k+l_k-(\xi+\beta);\varepsilon)\leq 1-\sigma\}.$$ Then, for $k\in \mathscr{A}^c(\lambda,\varepsilon)\cap \mathscr{B}^c(\lambda,\varepsilon)$, we have 
        \begin{align*}
            &\varphi(w_k+l_k-(\xi+\beta);\varepsilon)\\
            \geq &\varphi(w_k-\xi;\frac{\varepsilon}{2})\odot \varphi(l_k-\beta;\frac{\varepsilon}{2})\\
            >&(1-\lambda)\odot(1-\lambda)\\
            >&1-\sigma
        \end{align*}
        Therefore $$\mathscr{A}^c(\lambda,\varepsilon)\cap \mathscr{B}^c(\lambda,\varepsilon)\subset \{\alpha(n)+1\leq k\leq \vartheta(n): \varphi(w_k+l_k-(\xi+\beta);\varepsilon)> 1-\sigma\}$$ i.e., $$\{\alpha(n)+1\leq k\leq \vartheta(n): \varphi(w_k+l_k-(\xi+\beta);\varepsilon)\leq 1-\sigma\}\subset\mathscr{A}(\lambda,\varepsilon)\cup\mathscr{B}(\lambda,\varepsilon).$$ Consequently, $\delta_\alpha^\vartheta(\mathscr{C}(\sigma,\varepsilon))=0$ i.e. $\mathscr{D}_\alpha^\vartheta(\mathcal{S}^\varphi)-\lim w_k+l_k=\xi+\beta$.
    \item If $\kappa=0$, then it is obvious. So, let $\kappa\neq 0$. Suppose $\mathscr{D}_\alpha^\vartheta(\mathcal{S}^\varphi)-\lim w_k=\xi$. Then for every $\varepsilon>0$ and $\sigma\in(0,1)$, $\delta_\alpha^\vartheta(\{\alpha(n)+1\leq k\leq \vartheta(n): \varphi(w_k-\xi;\frac{\varepsilon}{|c|})\leq 1-\sigma \})=0$. Since $$\{\alpha(n)+1\leq k\leq \vartheta(n): \varphi(\kappa w_k-\kappa\xi;\varepsilon)\leq 1-\sigma \}=\{\alpha(n)+1\leq k\leq \vartheta(n): \varphi(w_k-\xi;\frac{\varepsilon}{\vert \kappa \vert})\leq 1-\sigma \}$$ therefore $$\delta_\alpha^\vartheta(\{\alpha(n)+1\leq k\leq \vartheta(n): \varphi(\kappa w_k-\kappa\xi;\varepsilon)\leq 1-\sigma \})=0.$$ Hence $\mathscr{D}_\alpha^\vartheta(\mathcal{S}^\varphi)-\lim \kappa w_k=\kappa \xi$.
    \end{enumerate} 
\end{proof}

\begin{theorem}
    Let $\{w_k\}$ be a sequence in a PNS $\mathscr{X}$. If $w_k\xrightarrow{\varphi}\xi$ then $w_k\xrightarrow{\mathscr{D}_\alpha^\vartheta(\mathcal{S}^\varphi)}\xi$.
\end{theorem}

\begin{proof}
    Suppose that $w_k\xrightarrow{\varphi}\xi$. Then for every $\varepsilon>0$ and $\sigma\in(0,1)$, there is a positive integer $n_0$ such that $\varphi(w_k-\xi;\varepsilon)>1-\sigma$ for all $k\geq n_0$. Then, clearly the set  $\{\alpha(n)+1\leq k\leq \vartheta(n): \varphi(w_k-\xi;\varepsilon)\leq 1-\sigma \}$ contains at most finite number of terms. Consequently, $\delta_\alpha^\vartheta(\{\alpha(n)+1\leq k\leq \vartheta(n): \varphi(w_k-\xi;\varepsilon)\leq 1-\sigma \})=0$. This gives $w_k\xrightarrow{\mathscr{D}_\alpha^\vartheta(\mathcal{S}^\varphi)}\xi$. This completes the proof.
\end{proof}

But the converse of the above theorem need not be true which can be shown by the following example.
\begin{example}
    Let $\mathscr{X}$ be a real normed space with the usual norm. Define the continuous $t$-norm $\nu_1\odot \nu_2=\nu_1\nu_2$ for all $\nu_1,\nu_2\in[0,1]$. We take $\varphi_0(\tau;\varepsilon)=\frac{\varepsilon}{\varepsilon+\vert\tau\vert}$ for $\tau\in\mathscr{X}, \varepsilon>0$. Then $\varphi_0$ is a probabilistic norm on $\mathscr{X}$. Let us define the sequence $\{w_k\}$ in $\mathscr{X}$ as $w_k=\begin{cases}
        1, \ \text{if}\ k=i^2,\ i\in\mathbb{N}\\
        0, \ \text{otherwise}
    \end{cases}$. Then for every $\varepsilon>0$ and $\sigma\in(0,1)$ we have \begin{align*}
        \mathscr{A}&= \{\alpha(n)+1\leq k\leq \vartheta(n): \varphi_0(w_k;\varepsilon)\leq 1-\sigma \}\\
        &=\{ \alpha(n)+1\leq k\leq \vartheta(n): \frac{\varepsilon}{\varepsilon+\vert w_k\vert}\leq 1-\sigma\}\\
        &=\{ \alpha(n)+1\leq k\leq \vartheta(n): \vert w_k\vert\geq \frac{\varepsilon\sigma}{1-\sigma}>0\}\\
        &\subseteq \{\alpha(n)+1\leq k\leq \vartheta(n): w_k=1 \}\\
        &= \{\alpha(n)+1\leq k\leq \vartheta(n): k=i^2,i\in\mathbb{N}\}.
    \end{align*}
    Therefore, $\delta_\alpha^\vartheta(\mathscr{A})\leq \lim_{n\to\infty}\frac{\left[ \sqrt{\vartheta(n)}-\sqrt{\alpha(n)+1}\right]}{\vartheta(n)-\alpha(n)}\leq \lim_{n\to\infty}\frac{\sqrt{\vartheta(n)}-\sqrt{\alpha(n)}}{\vartheta(n)-\alpha(n)}=0$. Hence $w_k\xrightarrow{\mathscr{D}_\alpha^\vartheta(\mathcal{S}^\varphi)}0$. But the sequence is not convergent in the usual sense, so using Remark \ref{rem2.1} we can say the sequence is not convergent to zero with respect to $\varphi_0$.
\end{example}

\begin{theorem}\label{thm3.3}
    Let $\{w_k\}$ and $\{l_k\}$ be two sequences in a PNS $\mathscr{X}$. If $\delta_\alpha^\vartheta(\{k\in\mathbb{N}: w_k\neq l_k \})=0$ and $l_k\xrightarrow{\varphi}\xi$ then $w_k\xrightarrow{\mathscr{D}_\alpha^\vartheta(\mathcal{S}^\varphi)}\xi$.
\end{theorem}

\begin{proof}
    Suppose that $\delta_\alpha^\vartheta(\{k\in\mathbb{N}: w_k\neq l_k \})=0$ and $l_k\xrightarrow{\varphi}\xi$. Then for every $\varepsilon>0$ and $\sigma\in(0,1)$ there exists $n_0\in\mathbb{N}$ such that $\varphi(l_k-\xi;\varepsilon)>1-\sigma$ for all $k>n_0$. Then  $$\mathscr{A}=\{k\in\mathbb{N}: \varphi(l_k-\xi;\varepsilon)\leq 1-\sigma\}\subset\{1,2,\ldots,n_0\}.$$ So $\delta_\alpha^\vartheta(\mathscr{A})=0$.  Since $$\{\alpha(n)+1\leq k\leq \vartheta(n): \varphi(w_k-\xi;\varepsilon)\leq 1-\sigma \}\subset \
    \mathscr{A}\cap \{k\in\mathbb{N}: w_k\neq l_k \},$$  $$\delta_\alpha^\vartheta(\{\alpha(n)+1\leq k\leq \vartheta(n): \varphi(w_k-\xi;\varepsilon)\leq 1-\sigma \})=0.$$ Hence $w_k\xrightarrow{\mathscr{D}_\alpha^\vartheta(\mathcal{S}^\varphi)}\xi$. This completes the proof.
\end{proof}

\begin{theorem}\label{thm3.4}
     Let $\{w_k\}$ be a  sequence in a PNS $\mathscr{X}$ and   the sequence $\{ \frac{\alpha(n)}{\vartheta(n)-\alpha(n)}\}$ be bounded. If $\mathcal{S}^\varphi-\lim w_k=\xi$ then $\mathscr{D}_\alpha^\vartheta(\mathcal{S}^\varphi)-\lim w_k=\xi$.
\end{theorem}

\begin{proof}
    Given that $\lim_{n\to\infty}\vartheta(n)=\infty$. Suppose that $\mathcal{S}^\varphi-\lim w_k=\xi$. Let $\varepsilon>0$ and $\sigma\in(0,1)$. Then by the fact of Theorem 2.2.1 in \cite{Kucukaslan}, we have $$\lim_{n\to\infty}\frac{\vert\{k\leq \vartheta(n): \varphi(w_k-\xi;\varepsilon)\leq 1-\sigma\}\vert}{\vartheta(n)}=0.$$  Since $$\{\alpha(n)+1\leq k\leq \vartheta(n): \varphi(w_k-\xi;\varepsilon)\leq 1-\sigma\}\subset \{k\leq \vartheta(n): \varphi(w_k-\xi;\varepsilon)\leq 1-\sigma\},$$ we get   $$\vert\{\alpha(n)+1\leq k\leq \vartheta(n): \varphi(w_k-\xi;\varepsilon)\leq 1-\sigma\}\vert\leq \vert\{k\leq \vartheta(n): \varphi(w_k-\xi;\varepsilon)\leq 1-\sigma\}\vert.$$ Consequently \begin{align*}
        &\frac{1}{\vartheta(n)-\alpha(n)}\vert\{\alpha(n)+1\leq k\leq \vartheta(n): \varphi(w_k-\xi;\varepsilon)\leq 1-\sigma\}\vert\\
        \leq &\left(1+\frac{\alpha(n)}{\vartheta(n)-\alpha(n)}\right).\frac{1}{\vartheta(n)}\vert\{k\leq \vartheta(n): \varphi(w_k-\xi;\varepsilon)\leq 1-\sigma\}\vert.
    \end{align*} Since $\left\{ \frac{\alpha(n)}{\vartheta(n)-\alpha(n)}\right\}$ is bounded and $\left\{ \frac{\vert\{k\leq \vartheta(n): \varphi(w_k-\xi;\varepsilon)\leq 1-\sigma\}\vert}{\vartheta(n)}\right\}$ is convergent to zero as $n\to\infty$, therefore $\lim_{n\to\infty}\frac{1}{\vartheta(n)-\alpha(n)}\vert\{\alpha(n)+1\leq k\leq \vartheta(n): \varphi(w_k-\xi;\varepsilon)\leq 1-\sigma\}\vert=0$. This shows $\mathscr{D}_\alpha^\vartheta(\mathcal{S}^\varphi)-\lim w_k=\xi$. This completes the proof.
\end{proof}

\begin{theorem}
    Let $\{w_k\}$  be a sequence in a PNS $\mathscr{X}$ and $\varrho(n)$ and $\varsigma(n)$ be sequences of  non negative integers such that $\alpha(n)\leq \varrho(n)<\varsigma(n)\leq \vartheta(n)$ for all $n\in\mathbb{N}$. Also, let the sets $\{k\in\mathbb{N}: \alpha(n)<k\leq \varrho(n)\}$ and $\{k\in\mathbb{N}: \varsigma(n)<k\leq \vartheta(n) \}$ be finite. If $\mathscr{D}_\varrho^\varsigma(\mathcal{S}^\varphi)-\lim w_k=\xi$ then $\mathscr{D}_\alpha^\vartheta(\mathcal{S}^\varphi)-\lim w_k=\xi$.
\end{theorem}

\begin{proof}
    Suppose that $\mathscr{D}_\varrho^\varsigma(\mathcal{S}^\varphi)-\lim w_k=\xi$. Then for every $\varepsilon>0$ and $\sigma\in(0,1)$, \begin{equation}\label{eqn3.1}
        \delta_\varrho^\varsigma(\{\varrho(n)+1\leq k\leq \varsigma(n): \varphi(w_k-\xi;\varepsilon)\leq 1-\sigma\})=0.
    \end{equation}
    Given that \begin{equation}\label{eqn3.2}
        \{k\in\mathbb{N}: \alpha(n)<k\leq \varrho(n)\}\ \text{and}\ \{k\in\mathbb{N}: \varsigma(n)<k\leq \vartheta(n) \}
    \end{equation}
    are finite.
    Since \begin{align*}
    &\{\alpha(n)+1\leq k\leq \vartheta(n): \varphi(w_k-\xi;\varepsilon)\leq 1-\sigma\}\\
    =& \{\alpha(n)+1\leq k\leq \varrho(n): \varphi(w_k-\xi;\varepsilon)\leq 1-\sigma\}\\
    \cup & \{\varrho(n)+1\leq k\leq \varsigma(n): \varphi(w_k-\xi;\varepsilon)\leq 1-\sigma\}\\
    \cup &\{\varsigma(n)+1\leq k\leq \vartheta(n): \varphi(w_k-\xi;\varepsilon)\leq 1-\sigma\},
\end{align*}
We have 
\begin{align*}
    &\delta_\alpha^\vartheta(\{\alpha(n)+1\leq k\leq \vartheta(n): \varphi(w_k-\xi;\varepsilon)\leq 1-\sigma\})\\
    \leq & \delta_\alpha^\varrho(\{\alpha(n)+1\leq k\leq \varrho(n): \varphi(w_k-\xi;\varepsilon)\leq 1-\sigma\})\\
    + & \delta_\varrho^\varsigma(\{\varrho(n)+1\leq k\leq \varsigma(n): \varphi(w_k-\xi;\varepsilon)\leq 1-\sigma\})\\
    + &\delta_\varsigma^\vartheta(\{\varsigma(n)+1\leq k\leq \vartheta(n): \varphi(w_k-\xi;\varepsilon)\leq 1-\sigma\}).
\end{align*}
 Deferred density of a finite set being zero,  using the Equations \ref{eqn3.1} and \ref{eqn3.2} we get $\delta_\alpha^\vartheta(\{\alpha(n)+1\leq k\leq \vartheta(n): \varphi(w_k-\xi;\varepsilon)\leq 1-\sigma\})=0$ i.e. $\mathscr{D}_\alpha^\vartheta(\mathcal{S}^\varphi)-\lim w_k=\xi$. This completes the proof.
\end{proof}

\begin{theorem}
    Let $\{w_k\}$  be a sequence in a PNS $\mathscr{X}$ and $\varrho(n)$ and $\varsigma(n)$ be sequences of  non negative integers such that $\alpha(n)\leq \varrho(n)<\varsigma(n)\leq \vartheta(n)$ for all $n\in\mathbb{N}$. If $\lim_{n\to\infty}\frac{\vartheta(n)-\alpha(n)}{\varsigma(n)-\FancyVerbLineautorefname(n)}=\beta>0$ and $\mathscr{D}_\alpha^\vartheta(\mathcal{S}^\varphi)-\lim w_k=\xi$ then $\mathscr{D}_\varrho^\varsigma(\mathcal{S}^\varphi)-\lim w_k=\xi$.
\end{theorem}

\begin{proof}
    Suppose that $\mathscr{D}_\alpha^\vartheta(\mathcal{S}^\varphi)-\lim w_k=\xi$. Then for every $\varepsilon>0$ and $\sigma\in(0,1)$, \begin{equation}\label{eqn3.3}
        \lim_{n\to\infty}\frac{1}{\vartheta(n)-\alpha(n)}\vert\{\alpha(n)+1\leq k\leq \vartheta(n): \varphi(w_k-\xi;\varepsilon)\leq 1-\sigma \}\vert=0.
    \end{equation}
    It is clear that $\{\varrho(n)+1\leq k\leq \varsigma(n): \varphi(w_k-\xi;\varepsilon)\leq 1-\sigma \}\subset \{\alpha(n)+1\leq k\leq \vartheta(n): \varphi(w_k-\xi;\varepsilon)\leq 1-\sigma \}$. So, we get \begin{align*}
        &\vert\{\varrho(n)+1\leq k\leq \varsigma(n): \varphi(w_k-\xi;\varepsilon)\leq 1-\sigma \}\vert\\
        \leq &\vert\{\alpha(n)+1\leq k\leq \vartheta(n): \varphi(w_k-\xi;\varepsilon)\leq 1-\sigma \}\vert\\
        \implies &\frac{1}{\varsigma(n)-\varrho(n)}\vert\{\varrho(n)+1\leq k\leq \varsigma(n): \varphi(w_k-\xi;\varepsilon)\leq 1-\sigma \}\vert\\
        \leq & \frac{\vartheta(n)-\alpha(n)}{\varsigma(n)-\varrho(n)}.\frac{1}{\vartheta(n)-\alpha(n)}\vert\{\alpha(n)+1\leq k\leq \vartheta(n): \varphi(w_k-\xi;\varepsilon)\leq 1-\sigma \}\vert.
    \end{align*}
     Letting $n\to\infty$ on both sides and using (\ref{eqn3.3}) and our assumption $\lim_{n\to\infty}\frac{\vartheta(n)-\alpha(n)}{\varsigma(n)-\varrho(n)}=\beta>0$, we get $\lim_{n\to\infty}\frac{1}{\varsigma(n)-\varrho(n)}\vert\{\varrho(n)+1\leq k\leq \varsigma(n): \varphi(w_k-\xi;\varepsilon)\leq 1-\sigma \}\vert=0$. Hence $\mathscr{D}_\varrho^\varsigma(\mathcal{S}^\varphi)-\lim w_k=\xi$. This completes the proof.
\end{proof}

\begin{definition}
      Let $\{w_k\}$  be a sequence in a PNS $\mathscr{X}$. Then $\{w_k\}$ is named to be deferred statistically Cauchy sequence  if for every $\varepsilon>0$ and $\sigma\in(0,1)$ there is $n_0=n_0(\sigma)$ such that $\delta_\alpha^\vartheta(\{\alpha(n)+1\leq k\leq \vartheta(n): \varphi(w_k-w_{n_0};\varepsilon)\leq 1-\sigma \})=0$.
\end{definition}

\begin{theorem}
    Let $\{w_k\}$  be a sequence in a PNS $\mathscr{X}$. If  $\{w_k\}$ is $\mathscr{D}_\alpha^\vartheta(\mathcal{S}^\varphi)$-convergent then it is $\mathscr{D}_\alpha^\vartheta(\mathcal{S}^\varphi)$-Cauchy sequence.
\end{theorem}

\begin{proof}
    Suppose $\{w_k\}$ is $\mathscr{D}_\alpha^\vartheta(\mathcal{S}^\varphi)$-convergent to $\xi$ and $\sigma\in(0,1)$ be given. Choose $\lambda\in(0,1)$ such that $(1-\lambda)\odot(1-\lambda)>1-\sigma$. Then for every $\varepsilon>0$, $\delta_\alpha^\vartheta(\mathscr{A}(\lambda,\varepsilon))=0$ where $$\mathscr{A}(\lambda,\varepsilon)=\{\alpha(n)+1\leq k\leq \vartheta(n): \varphi(w_k-\xi;\frac{\varepsilon}{2})\leq 1-\lambda\}.$$ Then $\delta_\alpha^\vartheta(\mathbb{N}\setminus\mathscr{A}(\lambda,\varepsilon))=1$. So, there is $n_0\in\mathbb{N}\setminus\mathscr{A}(\lambda,\varepsilon)$. Then $$\varphi(w_{n_0}-\xi;\frac{\varepsilon}{2})>1-\lambda$$ holds good. Now, define $\mathscr{B}(\sigma,\varepsilon)=\{\alpha(n)+1\leq k\leq \vartheta(n): \varphi(w_k-w_{n_0};\varepsilon)\leq 1-\sigma\}$. It is sufficient to prove that $\mathscr{B}(\sigma,\varepsilon)\subset\mathscr{A}(\lambda,\varepsilon)$. Let $m\in\mathscr{B}(\sigma,\varepsilon)$. Then $$\varphi(w_m-w_{n_0};\varepsilon)\leq 1-\sigma.$$ We show $\varphi(w_m-\xi;\frac{\varepsilon}{2})\leq 1-\lambda$. If possible,  let $\varphi(w_m-\xi;\frac{\varepsilon}{2})>1-\lambda$. Then, we have 
    \begin{align*}
        1-\sigma&\geq \varphi(w_m-w_{n_0};\varepsilon)\\
        &\geq \varphi(w_{n_0}-\xi;\frac{\varepsilon}{2})\odot \varphi(w_m-\xi;\frac{\varepsilon}{2})\\
        &>(1-\lambda)\odot(1-\lambda)>1-\sigma
    \end{align*}
    which is not possible. Therefore $\varphi(w_m-\xi;\frac{\varepsilon}{2})\leq 1-\lambda$. Hence $m\in\mathscr{A}(\lambda,\varepsilon)$ i.e., $\mathscr{B}(\sigma,\varepsilon)\subset\mathscr{A}(\lambda,\varepsilon)$. This completes the proof.
\end{proof}

\subsection*{Conclusion and future developments}
Here, we have discussed the notion of deferred statistical convergence of sequences as a generalization of statistical, $\lambda$-statistical and lacunary statistical convergence of sequences with respect to probabilistic norm. In Theorem \ref{thm3.3} and \ref{thm3.4} the connection between  convergence method using deferred density and usual convergence method with regards to probabilistic norm has been established. Latter on, one may apply the notion of double sequences and sequences of sets on this convergence method for further developments and give the idea of deferred statistical boundedness to do useful connection with deferred statistical convergent sequences with regard to probabilistic norm.

\end{document}